%% file: aper-friese06.tex
\documentclass[12pt,a4paper,reqno]{amsart}
\usepackage[english]{babel}
\usepackage[utf8]{inputenc}
\usepackage{amsmath,amsfonts,amssymb,amsthm}
\usepackage{a4wide}
\usepackage{hyperref}
\usepackage{graphicx}
\usepackage{dsfont}
\usepackage{color}
\usepackage{enumitem}
\usepackage{tikz}
\usetikzlibrary{calc,angles,positioning,quotes,decorations,graphs}
\usetikzlibrary{shapes.geometric}
\usepackage{tkz-euclide}

\newtheorem{thm}{Theorem}[section]

\newtheorem{cor}[thm]{Corollary}

\newtheorem*{thm*}{Theorem}
\theoremstyle{definition}

\theoremstyle{remark}

\newcommand{\Z}{\mathbb{Z}}

\renewcommand{\phi}{\varphi}
\renewcommand{\rho}{\varrho}
\renewcommand{\epsilon}{\varepsilon}

\begin{document}
\title{Frieze patterns and aperiodic tilings of the plane}

\author{Dirk Frettl\"oh}
\address{Technische Fakult\"at, Bielefeld University\newline\hspace*{\parindent}Postfach 100131, 33501 Bielefeld, Germany}
\email{dfrettloeh@techfak.de}

\author{Jan Maz\'{a}\v{c}}
\address{Department f\"{u}r Mathematik, Friedrich-Alexander-Universit\"{a}t Erlangen-N\"{u}rnberg, \newline 
Cauerstra{\ss}e 11, 91058 Erlangen, Germany}
\email{mazacj@math.fau.de}

\date{\today}

\begin{abstract}
This short note provides two examples of aperiodic frieze patterns of the plane, supported on the rhombic Penrose tiling and the Godr\`eche--Lan\c{c}on--Billard tiling. That is, we provide a decoration of their vertices with positive integers which satisfy the diamond rule, in analogy to the usual (in)finite frieze patterns as defined by Conway and Coxeter. 
\end{abstract}

\maketitle

\section{Frieze patterns in a nutshell}
Frieze patterns were introduced by Conway and Coxeter in \cite{Co71,CC73a,CC73b}. A  \emph{Conway--Coxeter frieze pattern} (or frieze pattern for short) is a grid of numbers consisting of a finite number of bi-infinite rows as shown below. Any frieze begins with a row of zeros followed by a~row of ones. The first non-trivial row is called the \emph{quiddity row}. The second-to-last row consists again of ones, and the last row is a row of zeros.  So in general, a Conway--Coxeter frieze of width $m$ is of the form 
\begin{small}
    \begin{center}
        \begin{tabular}{ccccccccccccccc}
        $\cdots$ & 0 && 0 && 0 && 0 && 0 && 0 & $\cdots$ \\[3pt]
        $\cdots$ && 1 && 1 && 1 && 1 && 1 &&  $\cdots$ \\[3pt]
        $\cdots$ & $a^{}_{-1,1}$ && $a^{}_{0,2}$ && $a^{}_{1,3}$ && $a^{}_{2,4}$ && $a^{}_{3,5}$ && $a^{}_{4,6}$ & $\cdots$ \\[3pt]
        $\cdots$ && $a^{}_{-1,2}$ && $a^{}_{0,3}$ && $a^{}_{1,4}$ && $a^{}_{2,5}$ && $a^{}_{3,6}$ &&   $\cdots$ \\[3pt]
         & $\ddots$ && $\ddots$ && $\ddots$ && $\ddots$ && $\ddots$ &&  $\ddots$ & \\[3pt]
        $\cdots$ && $a^{}_{-2,m-1}$ && $a^{}_{-1,m}$ && $a^{}_{0,m+1}$ && $a^{}_{1,m+2}$ && $a^{}_{2,m+3}$ &&   $\cdots$ \\[3pt]
        $\cdots$ & 1 && 1 && 1 && 1 && 1 && 1 & $\cdots$ \\[3pt]
        $\cdots$ && 0 && 0 && 0 && 0 && 0 &&  $\cdots$ 
        \end{tabular}
    \end{center}
    \end{small}
and all its non-trivial entries are positive integers which satisfy the \emph{diamond rule}. This rule requires that for every diamond of the form 
\begin{center}
    \begin{tabular}{ccc}
          & $a$ & \\
          $b$ & & $c$ \\
          & $d$ &
    \end{tabular}
\end{center}      
holds that $bc-ad=1$. 

Such friezes are always periodic with period $m-3$ (meaning that every row is a periodic sequence), and the quiddity row has a surprising geometric interpretation in terms of triangulations of polygons, which is nicely summarised in the famous Conway--Coxeter theorem. 
\begin{thm*}[{\cite{CC73a,CC73b}}]
There is a bijection between triangulated polygons with $m$ vertices and Conway--Coxeter frieze patterns of width $m-3$. In particular, %
\begin{itemize}
    \item[(i)] Let $F$ be a Conway--Coxeter frieze of width $m-3$ with entries $a^{}_{i,j}$ as above. Then, the sequence $(a^{}_{0,2}, a^{}_{1,3}, \, \dots, \, , \, a^{}_{m-1,m+1})$ gives a triangulation of an $m$-gon with vertices $1,\, 2,\, \dots \, , \, m$ such that $a^{}_{i-1,i+1}$ denotes the number of triangles sharing the vertex $i$.
    \item[(ii)] Conversely, given a triangulated $m$-gon, one defines $a^{}_{i}$ to be the number of triangles touching the vertex $i$ and $(a^{}_{n})^{}_{n\in\Z}$ to be the periodic continuation (in both directions) of  $a^{}_1, \, a^{}_{2}, \, \dots\, , \, a^{}_{m}$. Then, $(a^{}_{n})^{}_{n\in\Z}$ is a quiddity row of a Conway--Coxeter frieze, which is defined by it and by the diamond rule.
\end{itemize}
\end{thm*}
It is perhaps not too surprising that the other rows of a Conway--Coxeter frieze also have a combinatorial interpretation, see for instance \cite{BCI74}. 

One can relax the definition of the Conway--Coxeter frieze and allow for only one row of zeros and let the width be infinite. In such a case, one speaks about an \emph{infinite frieze}. These were studied in connection with triangulations of punctured discs in \cite{Tsch15}, and later studied and characterised in \cite{BPT16}. One can show that \emph{every} quiddity row whose entries are all greater than one gives rise to a valid infinite frieze (meaning that the diamond rule always has some positive integer solution). For any such frieze, there is also a geometric interpretation in terms of triangulations of various surfaces. The closest case to the Conway--Coxeter frieze is the case when the quiddity row is periodic. Denote by $A^{}_{n,m}$ the annulus with $n$ marked points on its outer boundary and with $m$ marked points on its inner boundary. Then, the following theorem holds. 
\begin{thm*}[{\cite[Thms.~4.6~and~3.7\,]{BPT16}}]
Let $F$ be a periodic frieze with periodic quiddity row $a^{}_1, \, a^{}_{2}, \, \dots\, , \, a^{}_{n}$. Then either $F$ is finite or $a^{}_1, \, a^{}_{2}, \, \dots\, , \, a^{}_{m}$ defines a triangulation of $A^{}_{n,m}$ for some $m\geqslant 0$. Conversely, every triangulation of $A^{}_{n,m}$ defines a periodic infinite frieze. 
\end{thm*}
Even for the non-periodic quiddity sequences there exists a geometric interpretation of the corresponding infinite frieze in terms of triangulations of an infinite strip and vice versa \cite[Thms.\ 5.2 and 5.6\,]{BPT16}. Moreover, this geometric approach allows for understanding the other entries of the frieze in terms of geometric objects in both the periodic and the non-periodic case. 

Moreover, infinite friezes were studied intensiively in context of category theory and
representation theory, see for instance \cite{BM12}, or \cite{BFGST18} and references therein.
An example of a recent result connects representation theory of triangulated $\infty$-gons 
(the so-called fountain triangulations), tilting subcategories of the Frobenius category 
$\mathcal{C}^{}_2$ and, surprisingly, all Penrose tilings parametrised by sequences of zeros 
and ones, see Theorem A in \cite{EF25} for more details. 

In all the results mentioned above, the frieze patterns are always supported on a square grid; or, more
generally, a rhombic grid, see Figure \ref{fig:per_vs_aper} left. 
Considered as a tiling, one may view this grid as a periodic rhomb tiling of the plane. Of course, a finite frieze pattern 
is only a tiling of some strip in such a rhombic tiling. But if we allow the labels to be 
arbitrary integers, rather than positive integers, we can extend the frieze pattern throughout the entire 
tiling, still obeying the diamond rule: on top of the top row of zeroes, the diamond rule requires a row 
of $-1$s. On top of this row there is some freedom. So we may add a row $\ldots, -a_{0,2}, -a_{1,3}, 
-a_{2,4} , \ldots$ on top of the $-1$s. From there on, the diamond rule forces a copy of the original
frieze, but with all values negative, up to a row of zeroes again. From there on we may stack
the frieze in a periodic manner, alternating between positive values and negative values. 
In the same way, the frieze continues on the bottom of the finite frieze. 

Because of this observation, we decided to study whether one can decorate the vertices of other rhomb 
tilings in the plane with integers, or even positive integers, such that the values satisfy the diamond 
rule along every tile. We went for \emph{aperiodic} tilings of the plane (with rhombic tiles).
In this way, an entirely new class of frieze patterns is defined whose underlying geometry is no longer 
periodic, but aperiodic (in the strong sense), as illustrated in Figure~\ref{fig:per_vs_aper}, right.
\begin{figure}
    \centering
    \includegraphics[width=0.95\linewidth]{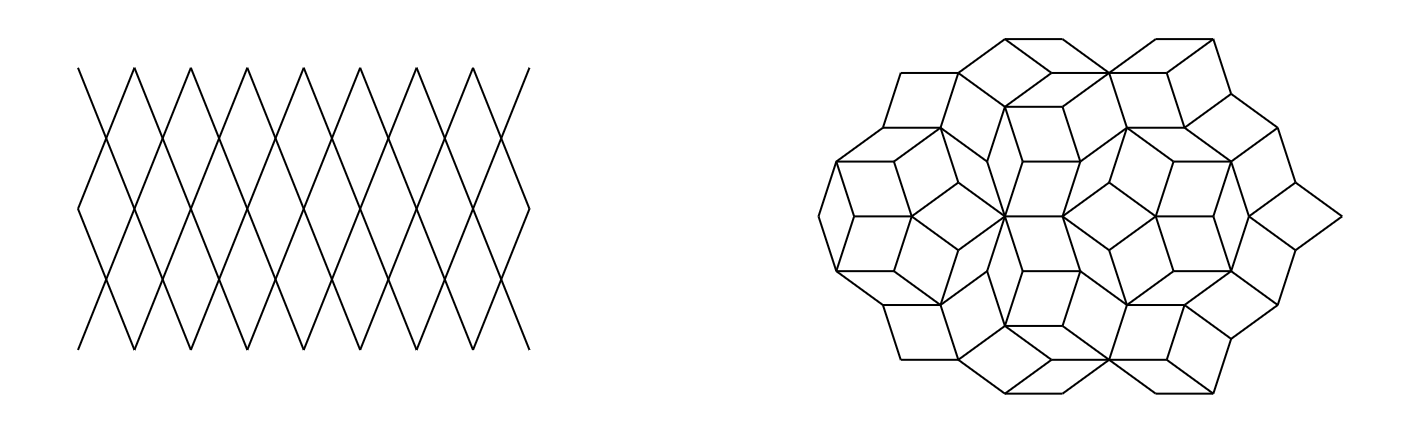}
    \caption{Periodic pattern on the left and a patch of Penrose rhombic tiling.}
    \label{fig:per_vs_aper}
\end{figure}
In this short work, we show that it is possible to define a frieze pattern in such a way on two examples of aperiodic tilings with rhombuses as tiles, namely, the famous Penrose rhombic tiling and the Godr\`eche--Lan\c{c}on--Billard tiling. It is a priori not clear at all if the underlying geometry even allows for the existence of a frieze pattern on a finite patch, and if so, whether such a pattern can be extended to the entire plane. We show two different approaches to prove that the two mentioned tilings admit an infinite family of frieze patterns on the entire tiling. 

\section{Penrose tiling}
The Penrose rhombic tiling is an aperiodic tiling of a plane by two rhombuses, thick and thin (and their rotated versions). The tiling was constructed / discovered by Sir Roger Penrose in 1974 \cite{Pen74}, and still serves as a prototype of an aperiodic tiling. A finite patch of this tiling is depicted in Figure \ref{fig:per_vs_aper} (right). In fact, there is not one single Penrose tiling, but uncountably many of them, but all are locally indistinguishable --- that is, each local patch of one of them occurs in every Penrose tiling.

There are many ways to construct Penrose tilings (matching rules, de Bruijn's pentagrid method, substitution rule), see for example \cite{TAO} for a survey. We will briefly recall a construction that gives us the desired frieze pattern on the vertices of the tiling (almost) for free. 

From the geometric viewpoint, one of the most natural ways to construct a Penrose rhombic tiling is to use the dualisation method for the root lattice $A^{}_{4}$ and the corresponding dual Voronoi cells. This lattice encapsulates the desired fivefold symmetry, and a projection to a suitable two-dimensional plane preserves the pentagonal symmetry. The construction in its full generality and with all details can be found in \cite{BKSZ90}. For a shorter summary including all the necessary steps, see \cite{Maz23}. 

Using this construction, the tiles of the Penrose tiling are projections of facets of regular polyhedra (forming a layer in four-dimensional space). All vertices of these polyhedra have their coordinates in the dual lattice $A^{\ast}_4$, more precisely: in $A^{\ast}_4 \backslash A^{}_{4}$ (the centres of those polyhedra are at lattice points of $A^{}_{4}$). The lattice $A^{\ast}_4$ is generated by 4 vectors $\mathbf{a}^{}_1,\mathbf{a}^{}_2,\mathbf{a}^{}_3,\mathbf{a}^{}_4$ whose projection into the plane where the tiling lives is shown in Figure~\ref{fig:basis_projections}. 
\begin{figure}[h]
\centering
\begin{tikzpicture}[yscale=0.8,xscale=0.8]
	\filldraw[black] (0,0) circle (2pt) node[above right]{$\ \scriptstyle 0$}; 
	\node[regular polygon, regular polygon sides=5, shape border rotate = -18,inner sep=1cm] (s) at (0,0) {} ;
	\draw[-stealth] (s.center) -- (s.corner 5) node [above]{$\scriptstyle \pi(\mathbf{a}^{}_1)$};
	\draw[-stealth] (s.center) -- (s.corner 1) node [right]{ $\scriptstyle \pi(\mathbf{a}^{}_2)$};
	\draw[-stealth] (s.center) -- (s.corner 2) node [above right]{$\scriptstyle \pi(\mathbf{a}^{}_3)$};
	\draw[-stealth] (s.center) -- (s.corner 3) node [below right]{$\scriptstyle \pi(\mathbf{a}^{}_4)$};
\end{tikzpicture}
\caption{Projections of the generators $\mathbf{a}^{}_i$ into the tiling plane.  }
\label{fig:basis_projections}
\end{figure}
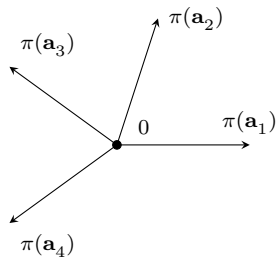

It follows that every vertex $\mathbf{v}$ of a~Penrose tiling can be obtained as a projection of a point $\mathbf{V} \in A^{\ast}_4\backslash A^{}_4$ and 
\[\mathbf{V} \, = \, \sum_{i=1}^4 n^{}_i \mathbf{a}^{}_i, \qquad n^{}_{i}\in \Z\, , \]
thus giving $\mathbf{v} \, = \, \sum_{i=1}^4 n^{}_i \pi(\mathbf{a}^{}_i)$. Since $\mathbf{V} \in A^{\ast}_4\backslash A^{}_4$, we always have $\sum_i n^{}_{i} \not\equiv 0 \mod 5$, see \cite[page~5]{Maz23}. We can then assign to each vertex $\mathbf{v}=\sum_{i=1}^4 n^{}_i \pi(\mathbf{a}^{}_i)$ the evaluation $r(\mathbf{v}) = \sum_i n^{}_{i} \mod 5$. We note that the function $r$ corresponds to the index function in the de Bruijn construction of Penrose tilings \cite{dBr81}.

The function effectively distinguishes different types of points; in particular (using the terminology from \cite{CS99}), it detects the shallow holes (if $r(\mathbf{v}) \equiv \pm 1\mod 5$) and deep holes (if $r(\mathbf{v}) \equiv \pm 2\mod 5$). One can easily check that every tile has exactly one vertex being a shallow hole. This vertex naturally determines the orientation of the tile and allows us to define the diamond rule as Figure~\ref{fig:tiles_deco} indicates. 
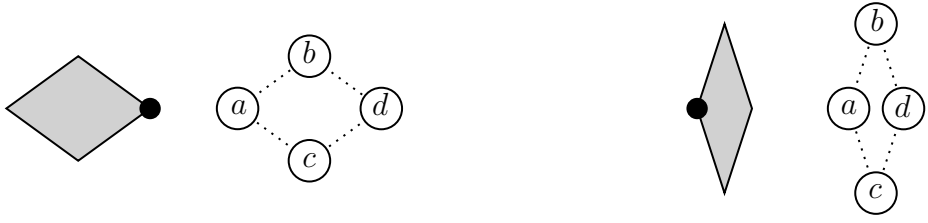
\begin{figure}[h!]
\centering
\tikzset{every picture/.style={line width=0.75pt}}
\begin{tikzpicture}[x=0.75pt,y=0.75pt,yscale=-1,xscale=1]
\draw  [color={rgb, 255:red, 0; green, 0; blue, 0 }  ,draw opacity=1 ][fill={rgb, 255:red, 155; green, 155; blue, 155 }  ,fill opacity=0.46 ] (441,112.68) -- (427.21,155.06) -- (413.46,112.67) -- (427.25,70.29) -- cycle ;

\draw  [color={rgb, 255:red, 0; green, 0; blue, 0 }  ,draw opacity=1 ][fill={rgb, 255:red, 155; green, 155; blue, 155 }  ,fill opacity=0.46 ] (139.04,112.67) -- (102.97,138.84) -- (66.93,112.63) -- (103,86.45) -- cycle ;

\draw  [fill={rgb, 255:red, 0; green, 0; blue, 0 }  ,fill opacity=1 ] (134.19,112.67) .. controls (134.19,109.99) and (136.36,107.82) .. (139.04,107.82) .. controls (141.72,107.82) and (143.89,109.99) .. (143.89,112.67) .. controls (143.89,115.34) and (141.72,117.52) .. (139.04,117.52) .. controls (136.36,117.52) and (134.19,115.34) .. (134.19,112.67) -- cycle ;

\draw  [fill={rgb, 255:red, 0; green, 0; blue, 0 }  ,fill opacity=1 ] (408.61,112.67) .. controls (408.61,109.99) and (410.78,107.82) .. (413.46,107.82) .. controls (416.14,107.82) and (418.31,109.99) .. (418.31,112.67) .. controls (418.31,115.34) and (416.14,117.52) .. (413.46,117.52) .. controls (410.78,117.52) and (408.61,115.34) .. (408.61,112.67) -- cycle ;

\draw  [color={rgb, 255:red, 0; green, 0; blue, 0 }  ,draw opacity=1 ][dash pattern={on 0.84pt off 2.51pt}] (517,112.68) -- (503.21,155.06) -- (489.46,112.67) -- (503.25,70.29) -- cycle ;

\draw  [color={rgb, 255:red, 0; green, 0; blue, 0 }  ,draw opacity=1 ][dash pattern={on 0.84pt off 2.51pt}] (255.04,112.67) -- (218.97,138.84) -- (182.93,112.63) -- (219,86.45) -- cycle ;

\draw  [fill={rgb, 255:red, 255; green, 255; blue, 255 }  ,fill opacity=1 ] (244.63,112.67) .. controls (244.63,106.92) and (249.29,102.26) .. (255.04,102.26) .. controls (260.79,102.26) and (265.44,106.92) .. (265.44,112.67) .. controls (265.44,118.41) and (260.79,123.07) .. (255.04,123.07) .. controls (249.29,123.07) and (244.63,118.41) .. (244.63,112.67) -- cycle ;

\draw  [fill={rgb, 255:red, 255; green, 255; blue, 255 }  ,fill opacity=1 ] (208.59,86.45) .. controls (208.59,80.7) and (213.25,76.05) .. (219,76.05) .. controls (224.74,76.05) and (229.4,80.7) .. (229.4,86.45) .. controls (229.4,92.2) and (224.74,96.86) .. (219,96.86) .. controls (213.25,96.86) and (208.59,92.2) .. (208.59,86.45) -- cycle ;

\draw  [fill={rgb, 255:red, 255; green, 255; blue, 255 }  ,fill opacity=1 ] (172.52,112.63) .. controls (172.52,106.88) and (177.18,102.22) .. (182.93,102.22) .. controls (188.68,102.22) and (193.33,106.88) .. (193.33,112.63) .. controls (193.33,118.37) and (188.68,123.03) .. (182.93,123.03) .. controls (177.18,123.03) and (172.52,118.37) .. (172.52,112.63) -- cycle ;

\draw  [fill={rgb, 255:red, 255; green, 255; blue, 255 }  ,fill opacity=1 ] (208.56,138.84) .. controls (208.56,133.1) and (213.22,128.44) .. (218.97,128.44) .. controls (224.72,128.44) and (229.38,133.1) .. (229.38,138.84) .. controls (229.38,144.59) and (224.72,149.25) .. (218.97,149.25) .. controls (213.22,149.25) and (208.56,144.59) .. (208.56,138.84) -- cycle ;

\draw  [fill={rgb, 255:red, 255; green, 255; blue, 255 }  ,fill opacity=1 ] (506.59,112.68) .. controls (506.59,106.93) and (511.25,102.27) .. (517,102.27) .. controls (522.75,102.27) and (527.41,106.93) .. (527.41,112.68) .. controls (527.41,118.43) and (522.75,123.09) .. (517,123.09) .. controls (511.25,123.09) and (506.59,118.43) .. (506.59,112.68) -- cycle ;

\draw  [fill={rgb, 255:red, 255; green, 255; blue, 255 }  ,fill opacity=1 ] (492.84,70.29) .. controls (492.84,64.54) and (497.5,59.88) .. (503.25,59.88) .. controls (509,59.88) and (513.66,64.54) .. (513.66,70.29) .. controls (513.66,76.03) and (509,80.69) .. (503.25,80.69) .. controls (497.5,80.69) and (492.84,76.03) .. (492.84,70.29) -- cycle ;

\draw  [fill={rgb, 255:red, 255; green, 255; blue, 255 }  ,fill opacity=1 ] (479.05,112.67) .. controls (479.05,106.92) and (483.71,102.26) .. (489.46,102.26) .. controls (495.2,102.26) and (499.86,106.92) .. (499.86,112.67) .. controls (499.86,118.41) and (495.2,123.07) .. (489.46,123.07) .. controls (483.71,123.07) and (479.05,118.41) .. (479.05,112.67) -- cycle ;

\draw  [fill={rgb, 255:red, 255; green, 255; blue, 255 }  ,fill opacity=1 ] (492.8,155.06) .. controls (492.8,149.31) and (497.46,144.65) .. (503.21,144.65) .. controls (508.95,144.65) and (513.61,149.31) .. (513.61,155.06) .. controls (513.61,160.8) and (508.95,165.46) .. (503.21,165.46) .. controls (497.46,165.46) and (492.8,160.8) .. (492.8,155.06) -- cycle ;

\draw (214,133.4) node [anchor=north west][inner sep=0.75pt]    {$c$};
\draw (214,77.4) node [anchor=north west][inner sep=0.75pt]    {$b$};
\draw (178,107) node [anchor=north west][inner sep=0.75pt]    {$a$};
\draw (249,103.4) node [anchor=north west][inner sep=0.75pt]    {$d$};

\draw (498,149) node [anchor=north west][inner sep=0.75pt]    {$c$};
\draw (499,61.4) node [anchor=north west][inner sep=0.75pt]    {$b$};
\draw (484,107) node [anchor=north west][inner sep=0.75pt]    {$a$};
\draw (511,104.4) node [anchor=north west][inner sep=0.75pt]    {$d$};

\end{tikzpicture}

\caption{Penrose rhombuses with indicated shallow hole vertex, and the corresponding naming of the vertices which allows to define the diamond rule $bc-ad=1$.}
\label{fig:tiles_deco}
\end{figure}

With this rule, the function $r$ provides a valid frieze pattern of any Penrose rhombic tiling. For that, observe that the thick rhombus appears in ten different orientations. They split into two classes, which are related by a $\tfrac{2\pi}{10}$-rotation. Since $r$ cannot be zero on any vertex, we obtain only two possible decorations depicted in Figure \ref{fig:frieze_thick}.
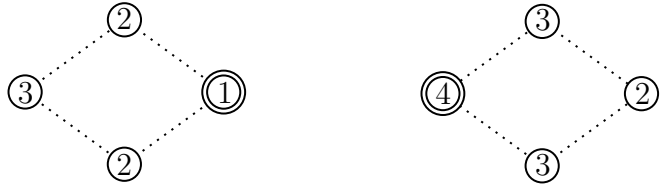
\begin{figure}[h!]
\centering
\tikzset{every picture/.style={line width=0.75pt}} 
\begin{tikzpicture}[x=0.6pt,y=0.6pt,yscale=-1,xscale=1]
\draw  [color={rgb, 255:red, 0; green, 0; blue, 0 }  ,draw opacity=1 ][dash pattern={on 0.84pt off 2.51pt}] (252.03,153.05) -- (189.74,198.26) -- (127.49,152.98) -- (189.79,107.78) -- cycle ;
\draw  [fill={rgb, 255:red, 255; green, 255; blue, 255 }  ,fill opacity=1 ] (179.33,198.26) .. controls (179.33,192.51) and (183.99,187.85) .. (189.74,187.85) .. controls (195.48,187.85) and (200.14,192.51) .. (200.14,198.26) .. controls (200.14,204) and (195.48,208.66) .. (189.74,208.66) .. controls (183.99,208.66) and (179.33,204) .. (179.33,198.26) -- cycle ;
\draw  [fill={rgb, 255:red, 255; green, 255; blue, 255 }  ,fill opacity=1 ] (238.59,153.05) .. controls (238.59,145.63) and (244.61,139.61) .. (252.03,139.61) .. controls (259.45,139.61) and (265.47,145.63) .. (265.47,153.05) .. controls (265.47,160.47) and (259.45,166.49) .. (252.03,166.49) .. controls (244.61,166.49) and (238.59,160.47) .. (238.59,153.05) -- cycle ;
\draw  [fill={rgb, 255:red, 255; green, 255; blue, 255 }  ,fill opacity=1 ] (241.62,153.05) .. controls (241.62,147.3) and (246.28,142.64) .. (252.03,142.64) .. controls (257.78,142.64) and (262.43,147.3) .. (262.43,153.05) .. controls (262.43,158.8) and (257.78,163.46) .. (252.03,163.46) .. controls (246.28,163.46) and (241.62,158.8) .. (241.62,153.05) -- cycle ;
\draw  [fill={rgb, 255:red, 255; green, 255; blue, 255 }  ,fill opacity=1 ] (179.38,107.78) .. controls (179.38,102.03) and (184.04,97.37) .. (189.79,97.37) .. controls (195.53,97.37) and (200.19,102.03) .. (200.19,107.78) .. controls (200.19,113.52) and (195.53,118.18) .. (189.79,118.18) .. controls (184.04,118.18) and (179.38,113.52) .. (179.38,107.78) -- cycle ;
\draw  [fill={rgb, 255:red, 255; green, 255; blue, 255 }  ,fill opacity=1 ] (117.09,152.98) .. controls (117.09,147.24) and (121.75,142.58) .. (127.49,142.58) .. controls (133.24,142.58) and (137.9,147.24) .. (137.9,152.98) .. controls (137.9,158.73) and (133.24,163.39) .. (127.49,163.39) .. controls (121.75,163.39) and (117.09,158.73) .. (117.09,152.98) -- cycle ;
\draw  [color={rgb, 255:red, 0; green, 0; blue, 0 }  ,draw opacity=1 ][dash pattern={on 0.84pt off 2.51pt}] (514.03,154.05) -- (451.74,199.26) -- (389.49,153.98) -- (451.79,108.78) -- cycle ;
\draw  [fill={rgb, 255:red, 255; green, 255; blue, 255 }  ,fill opacity=1 ] (441.33,199.26) .. controls (441.33,193.51) and (445.99,188.85) .. (451.74,188.85) .. controls (457.48,188.85) and (462.14,193.51) .. (462.14,199.26) .. controls (462.14,205) and (457.48,209.66) .. (451.74,209.66) .. controls (445.99,209.66) and (441.33,205) .. (441.33,199.26) -- cycle ;
\draw  [fill={rgb, 255:red, 255; green, 255; blue, 255 }  ,fill opacity=1 ] (376.06,153.98) .. controls (376.06,146.56) and (382.07,140.55) .. (389.49,140.55) .. controls (396.92,140.55) and (402.93,146.56) .. (402.93,153.98) .. controls (402.93,161.41) and (396.92,167.42) .. (389.49,167.42) .. controls (382.07,167.42) and (376.06,161.41) .. (376.06,153.98) -- cycle ;
\draw  [fill={rgb, 255:red, 255; green, 255; blue, 255 }  ,fill opacity=1 ] (503.62,154.05) .. controls (503.62,148.3) and (508.28,143.64) .. (514.03,143.64) .. controls (519.78,143.64) and (524.43,148.3) .. (524.43,154.05) .. controls (524.43,159.8) and (519.78,164.46) .. (514.03,164.46) .. controls (508.28,164.46) and (503.62,159.8) .. (503.62,154.05) -- cycle ;
\draw  [fill={rgb, 255:red, 255; green, 255; blue, 255 }  ,fill opacity=1 ] (441.38,108.78) .. controls (441.38,103.03) and (446.04,98.37) .. (451.79,98.37) .. controls (457.53,98.37) and (462.19,103.03) .. (462.19,108.78) .. controls (462.19,114.52) and (457.53,119.18) .. (451.79,119.18) .. controls (446.04,119.18) and (441.38,114.52) .. (441.38,108.78) -- cycle ;
\draw  [fill={rgb, 255:red, 255; green, 255; blue, 255 }  ,fill opacity=1 ] (379.09,153.98) .. controls (379.09,148.24) and (383.75,143.58) .. (389.49,143.58) .. controls (395.24,143.58) and (399.9,148.24) .. (399.9,153.98) .. controls (399.9,159.73) and (395.24,164.39) .. (389.49,164.39) .. controls (383.75,164.39) and (379.09,159.73) .. (379.09,153.98) -- cycle ;

\draw (183,190) node [anchor=north west][inner sep=0.75pt]    {$2$};
\draw (183,99) node [anchor=north west][inner sep=0.75pt]    {$2$};
\draw (121,145) node [anchor=north west][inner sep=0.75pt]    {$3$};
\draw (246,145) node [anchor=north west][inner sep=0.75pt]    {$1$};

\draw (445.5,190.3) node [anchor=north west][inner sep=0.75pt]    {$3$};
\draw (445.5,99.3) node [anchor=north west][inner sep=0.75pt]    {$3$};
\draw (383,145.5) node [anchor=north west][inner sep=0.75pt]    {$4$};
\draw (508,145.5) node [anchor=north west][inner sep=0.75pt]    {$2$};
\end{tikzpicture}
\caption{The two possible vertex decorations for the thick rhombuses. Each of the tiles appears in five different orientations. The double circle indicates the shallow hole.}
\label{fig:frieze_thick}
\end{figure}

The same holds for the thin rhombuses (Figure~\ref{fig:frieze_thin}), which also admit two possible decorations.
\begin{figure}[h!]
\centering
\tikzset{every picture/.style={line width=0.75pt}}
\begin{tikzpicture}[x=0.6pt,y=0.6pt,yscale=-1,xscale=1]
\draw  [color={rgb, 255:red, 0; green, 0; blue, 0 }  ,draw opacity=1 ][dash pattern={on 0.84pt off 2.51pt}] (233.47,171.67) -- (210.67,241.71) -- (187.94,171.64) -- (210.74,101.6) -- cycle ;
\draw  [color={rgb, 255:red, 0; green, 0; blue, 0 }  ,draw opacity=1 ][dash pattern={on 0.84pt off 2.51pt}] (451.99,171.69) -- (429.19,241.73) -- (406.47,171.67) -- (429.27,101.62) -- cycle ;
\draw  [fill={rgb, 255:red, 255; green, 255; blue, 255 }  ,fill opacity=1 ] (418.86,101.62) .. controls (418.86,95.88) and (423.52,91.22) .. (429.27,91.22) .. controls (435.01,91.22) and (439.67,95.88) .. (439.67,101.62) .. controls (439.67,107.37) and (435.01,112.03) .. (429.27,112.03) .. controls (423.52,112.03) and (418.86,107.37) .. (418.86,101.62) -- cycle ;
\draw  [fill={rgb, 255:red, 255; green, 255; blue, 255 }  ,fill opacity=1 ] (393.03,171.67) .. controls (393.03,164.24) and (399.05,158.23) .. (406.47,158.23) .. controls (413.89,158.23) and (419.91,164.24) .. (419.91,171.67) .. controls (419.91,179.09) and (413.89,185.1) .. (406.47,185.1) .. controls (399.05,185.1) and (393.03,179.09) .. (393.03,171.67) -- cycle ;
\draw  [fill={rgb, 255:red, 255; green, 255; blue, 255 }  ,fill opacity=1 ] (396.06,171.67) .. controls (396.06,165.92) and (400.72,161.26) .. (406.47,161.26) .. controls (412.21,161.26) and (416.87,165.92) .. (416.87,171.67) .. controls (416.87,177.41) and (412.21,182.07) .. (406.47,182.07) .. controls (400.72,182.07) and (396.06,177.41) .. (396.06,171.67) -- cycle ;
\draw  [fill={rgb, 255:red, 255; green, 255; blue, 255 }  ,fill opacity=1 ] (418.79,241.73) .. controls (418.79,235.98) and (423.45,231.33) .. (429.19,231.33) .. controls (434.94,231.33) and (439.6,235.98) .. (439.6,241.73) .. controls (439.6,247.48) and (434.94,252.14) .. (429.19,252.14) .. controls (423.45,252.14) and (418.79,247.48) .. (418.79,241.73) -- cycle ;
\draw  [fill={rgb, 255:red, 255; green, 255; blue, 255 }  ,fill opacity=1 ] (441.59,171.69) .. controls (441.59,165.94) and (446.24,161.28) .. (451.99,161.28) .. controls (457.74,161.28) and (462.4,165.94) .. (462.4,171.69) .. controls (462.4,177.44) and (457.74,182.1) .. (451.99,182.1) .. controls (446.24,182.1) and (441.59,177.44) .. (441.59,171.69) -- cycle ;
\draw  [fill={rgb, 255:red, 255; green, 255; blue, 255 }  ,fill opacity=1 ] (200.26,241.71) .. controls (200.26,235.96) and (204.92,231.3) .. (210.67,231.3) .. controls (216.42,231.3) and (221.07,235.96) .. (221.07,241.71) .. controls (221.07,247.45) and (216.42,252.11) .. (210.67,252.11) .. controls (204.92,252.11) and (200.26,247.45) .. (200.26,241.71) -- cycle ;
\draw  [fill={rgb, 255:red, 255; green, 255; blue, 255 }  ,fill opacity=1 ] (220.03,171.67) .. controls (220.03,164.24) and (226.05,158.23) .. (233.47,158.23) .. controls (240.89,158.23) and (246.91,164.24) .. (246.91,171.67) .. controls (246.91,179.09) and (240.89,185.1) .. (233.47,185.1) .. controls (226.05,185.1) and (220.03,179.09) .. (220.03,171.67) -- cycle ;
\draw  [fill={rgb, 255:red, 255; green, 255; blue, 255 }  ,fill opacity=1 ] (177.54,171.64) .. controls (177.54,165.9) and (182.2,161.24) .. (187.94,161.24) .. controls (193.69,161.24) and (198.35,165.9) .. (198.35,171.64) .. controls (198.35,177.39) and (193.69,182.05) .. (187.94,182.05) .. controls (182.2,182.05) and (177.54,177.39) .. (177.54,171.64) -- cycle ;
\draw  [fill={rgb, 255:red, 255; green, 255; blue, 255 }  ,fill opacity=1 ] (200.34,101.6) .. controls (200.34,95.85) and (205,91.19) .. (210.74,91.19) .. controls (216.49,91.19) and (221.15,95.85) .. (221.15,101.6) .. controls (221.15,107.35) and (216.49,112.01) .. (210.74,112.01) .. controls (205,112.01) and (200.34,107.35) .. (200.34,101.6) -- cycle ;
\draw  [fill={rgb, 255:red, 255; green, 255; blue, 255 }  ,fill opacity=1 ] (223.06,171.67) .. controls (223.06,165.92) and (227.72,161.26) .. (233.47,161.26) .. controls (239.21,161.26) and (243.87,165.92) .. (243.87,171.67) .. controls (243.87,177.41) and (239.21,182.07) .. (233.47,182.07) .. controls (227.72,182.07) and (223.06,177.41) .. (223.06,171.67) -- cycle ;

\draw (423,92.4) node [anchor=north west][inner sep=0.75pt]    {$2$};
\draw (423,232.4) node [anchor=north west][inner sep=0.75pt]    {$2$};
\draw (446,162.4) node [anchor=north west][inner sep=0.75pt]    {$3$};
\draw (400.5,163.4) node [anchor=north west][inner sep=0.75pt]    {$1$};

\draw (204.5,232.4) node [anchor=north west][inner sep=0.75pt]    {$3$};
\draw (204.5,92.4) node [anchor=north west][inner sep=0.75pt]    {$3$};
\draw (227,162.4) node [anchor=north west][inner sep=0.75pt]    {$4$};
\draw (182,162.4) node [anchor=north west][inner sep=0.75pt]    {$2$};
\end{tikzpicture}
\caption{The two possible vertex decorations for the thin rhombuses. Each of the tiles appears in five different orientations. The double circle indicates the shallow hole.}
\label{fig:frieze_thin}
\end{figure}
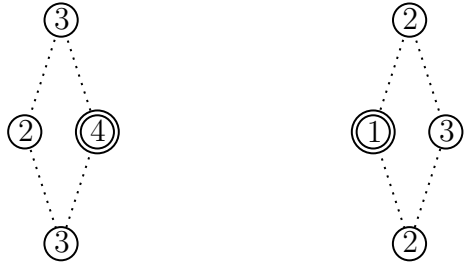

Since the function $r$ is well defined for every vertex of the Penrose tiling and since we assign to each  
tile a pattern satisfying the diamond rule, we obtain a frieze pattern on the entire Penrose tiling. Thus 
we obtain the following theorem. 
\begin{thm} \label{thm:frieze-pen}
Every Penrose rhombic tiling admits a frieze pattern on its vertices using only four distinct values.  \qed
\end{thm}
As a simple consequence, we also obtain the following.
\begin{cor} \label{cor:inf-pen}
Every Penrose rhombic tiling admits infinitely many frieze patterns using four distinct values.
\end{cor}
\begin{proof}
Since the constructed pattern on all rhombuses is always of the form
\begin{center}
    \begin{tabular}{ccc}
          & $n$ & \\
          $n+1$ & & $n+1$ \\
          & $n+2$ &
    \end{tabular}
\end{center}
with $n=1$ or $n=2$, and since $(n+1)^2-n(n+2)=1$ for any $n$, we can replace the numbers $\{1,2,3,4\}$ with $\{n,n+1,n+2,n+3 \}$ and still have a valid frieze pattern. In fact, we may as well choose $n$ to be an arbitrary real number and the diamond rule still holds.
\end{proof}
Figure \ref{fig:penr-frieze} shows the frieze pattern on a Penrose rhomb tiling, arising from the labelling
of the vertices of the Penrose tiling with respect to the construction above.
\begin{figure}
\[ \includegraphics[width=.8\textwidth]{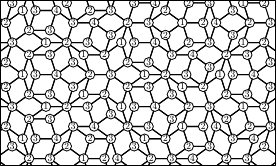} \]
\caption{ \label{fig:penr-frieze} A frieze pattern fulfilling the diamond rule superimposed on a~Penrose rhomb tiling. Image from \cite{TAO} with approval from the authors.}
\end{figure}

\section{Godr\`eche--Lan\c{c}on--Billard tiling}
Using the Penrose rhombuses, one can define another way of assembling them that results in a tiling of the plane whose properties are significantly different from those of the Penrose rhombic tiling. The tiling was first described in \cite{LB}, its substitution rule in \cite{GL}, and nowadays comes usually under the name Godr\`eche--Lan\c{c}on--Billard (GLB) tiling. It serves as an example of a tiling for which one does not have a description using the higher-dimensional structure and its projection, and whose diffraction is a singular continuous measure \cite{BGM19,Man19}, in complete contrast to Penrose rhombic tiling, which is pure-point diffractive. 
The combinatorics of the GLB tiling is also richer, and the tiling admits a more complicated atlas of vertex configurations (see \cite[Sec.~3.1.1]{Maz25}). 

This tiling can be constructed using the substitution rule in Figure \ref{fig:GLB1}.
\begin{figure}[h]
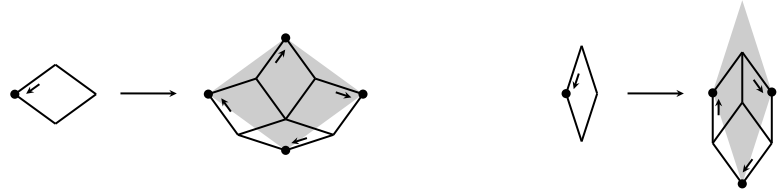

\centering
\include{glb1}
\caption{The action of the original GLB substitution rule on thick and thin rhombuses. The black points and arrows serve to mark the orientation of each tile.}
\label{fig:GLB1}
\end{figure}

One starts with a single tile and iteratively applies the rule to create larger and larger patches of tiles. These --- in the limit --- provide a tiling of the plane. We refer the reader to \cite{TAO} for further details on the construction of substitution tilings. We keep track of the black points which allow
us to define the diamond rule for both tile types as in the case of the Penrose tiling in the previous section.

The tiling is invariant under the fourth iteration of the substitution rule. Since the substitution rule is primitive \cite{TAO}, and the resulting tilings are aperiodic,  the resulting tilings have the unique decomposition property \cite{Sol}, meaning that any patch of the GLB tiling can be uniquely grouped into the supertiles (the inflated tiles according to the substitution rule at Figure \ref{fig:GLB1}). This is schematically depicted in Figure~\ref{fig:GLB2}.

\begin{figure}
    \centering
    \includegraphics[width=0.75\linewidth]{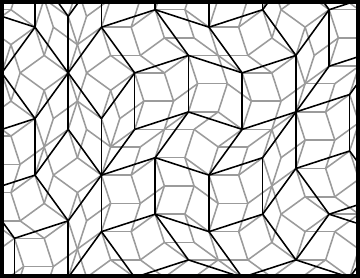}
    \caption{A patch of GLB tiling (grey) with its level-1 supertile structure (black).}
    \label{fig:GLB2}
\end{figure}

This property is the key to obtaining the frieze pattern on this tiling. Observe that this decomposition assigns to each tile one particular point (the black points in Figure \ref{fig:GLB1}). These are the vertices of the supertile structure. 
Let all black vertices (the supertile vertices) be labelled by 1. All vertices exactly one edge distant from a vertex
with label 1 will get label~2. The remaining vertices all get label 3. This labelling is indicated in Figure \ref{fig:GLB1}.
It is consistent with the diamond rule on the supertiles.

Since the GLB tiling can be grouped into supertiles, and each rhomb belongs to exactly one supertile, it follows 
immediately that the diamond rule is fulfilled throughout the entire tiling. It remains to convince oneself
that the labelling is consistent. This can be done in three simple steps: A supervertex is always labelled 1. 
Since every tile meets another one edge-to-edge in the GLB tilings, the supertiles also meet
(super-)edge-to-(super-)edge in the tiling, and supervertices are labelled 1 consistently. 
The vertex corresponding to the midpoint of a super-edge is always labelled 2 with respect to both supertiles
sharing this super-edge. The remaining vertices are labelled 3. By construction, they are always contained 
in the interior of some supertile, so no ambiguity can occur in this case.

\begin{thm}
Every GLB tiling admits a frieze pattern on its vertices using only three distinct values. \qed
\end{thm}
Again, since $n(n+2) = (n+1)^2-1$, we obtain the following result in exactly the same manner as 
Corollary \ref{cor:inf-pen}.
\begin{cor}
Every GLB tiling admits infinitely many frieze patterns using three distinct values only. \qed
\end{cor}
Figure \ref{fig:glb-fries} shows a patch of a GLB tiling with the frieze pattern constructed above.
\begin{figure}
\[ \includegraphics[width=.8\textwidth]{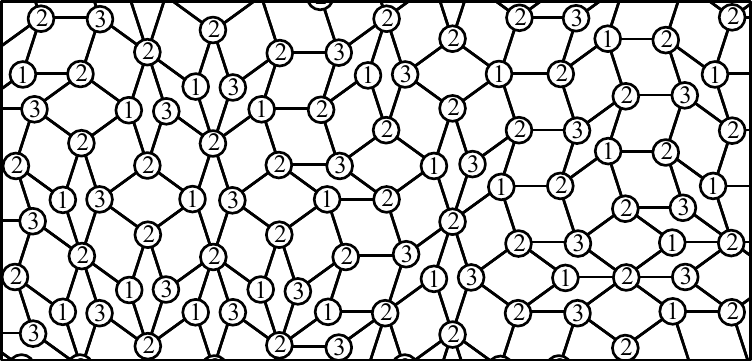} \]
\caption{ \label{fig:glb-fries} A frieze pattern fulfilling the diamond rule with labels $1,2,3$ 
superimposed on a GLB tiling.} 
\end{figure}

\section{Concluding remarks}
Here we introduced the first two examples of frieze patterns on aperiodic rhomb tilings. This
raises several questions: Are there frieze patterns on every aperiodic rhomb tiling? What about the Ammann--Beenker tiling? (We tried, without success.) Are there frieze patterns on the Penrose tiling where the number of distinct labels is not four? Similar for the GLB tiling, with the number of labels not being three. What are the requirements
for an aperiodic rhomb tiling to allow for a frieze pattern? What about other rhomb tilings,
for instance periodic ones other than the rhombic grid?

\section*{Acknowledgement}
Both authors thank Eleonora Faber for an introduction to infinite frieze patterns during a~stay at Mathematisches Forschungsinstitut Oberwolfach (MFO) in February 2026. The authors also greatly acknowledge the hospitality of the MFO, where the main ideas were discussed, and where most of the manuscript was written. JM was supported by MFO as the Leibniz Fellow (Project Nr 2603q).

\end{document}

%% file: glb1.tex
\tikzset{every picture/.style={line width=0.75pt}} 

\begin{tikzpicture}[x=0.53pt,y=0.53pt,yscale=-1,xscale=1]
\path (500,20); 

\draw  [draw opacity=0][fill={rgb, 255:red, 155; green, 155; blue, 155 }  ,fill opacity=0.5 ] (583.44,-56.12) -- (604.4,8.99) -- (583.44,74.1) -- (562.47,8.99) -- cycle ;
\draw    (583.44,74.1) -- (562.47,45.24) ;
\draw    (604.4,45.24) -- (583.44,74.1) ;
\draw    (583.44,-19.28) -- (562.47,9.58) ;
\draw    (562.47,9.58) -- (562.47,45.24) ;
\draw    (583.44,-19.28) -- (583.44,16.39) ;
\draw    (583.44,16.39) -- (562.47,45.24) ;
\draw    (583.44,16.39) -- (604.4,45.24) ;
\draw    (583.44,-19.28) -- (604.4,9.58) ;
\draw    (604.4,9.58) -- (604.4,45.24) ;
\draw  [fill={rgb, 255:red, 0; green, 0; blue, 0 }  ,fill opacity=1 ] (585.96,74.1) .. controls (585.96,75.49) and (584.83,76.62) .. (583.44,76.62) .. controls (582.04,76.62) and (580.91,75.49) .. (580.91,74.1) .. controls (580.91,72.7) and (582.04,71.57) .. (583.44,71.57) .. controls (584.83,71.57) and (585.96,72.7) .. (585.96,74.1) -- cycle ;
\draw  [fill={rgb, 255:red, 0; green, 0; blue, 0 }  ,fill opacity=1 ] (606.92,8.99) .. controls (606.92,10.38) and (605.79,11.51) .. (604.4,11.51) .. controls (603.01,11.51) and (601.88,10.38) .. (601.88,8.99) .. controls (601.88,7.6) and (603.01,6.47) .. (604.4,6.47) .. controls (605.79,6.47) and (606.92,7.6) .. (606.92,8.99) -- cycle ;
\draw  [fill={rgb, 255:red, 0; green, 0; blue, 0 }  ,fill opacity=1 ] (565,9.58) .. controls (565,10.97) and (563.87,12.1) .. (562.47,12.1) .. controls (561.08,12.1) and (559.95,10.97) .. (559.95,9.58) .. controls (559.95,8.18) and (561.08,7.05) .. (562.47,7.05) .. controls (563.87,7.05) and (565,8.18) .. (565,9.58) -- cycle ;
\draw  [draw opacity=0][fill={rgb, 255:red, 155; green, 155; blue, 155 }  ,fill opacity=0.5 ] (259.29,-29.42) -- (314.17,10.46) -- (259.29,50.33) -- (204.4,10.46) -- cycle ;
\draw    (469.41,-23.82) -- (458.38,10.1) ;
\draw    (458.38,10.1) -- (469.41,44.02) ;
\draw    (469.41,-23.82) -- (480.43,10.1) ;
\draw    (480.43,10.1) -- (469.41,44.02) ;

\draw [line width=0.75]    (95.85,-10.43) -- (67,10.54) ;
\draw [line width=0.75]    (124.71,10.54) -- (95.85,31.5) ;
\draw [line width=0.75]    (67,10.54) -- (95.85,31.5) ;
\draw [line width=0.75]    (95.85,-10.43) -- (124.71,10.54) ;

\draw    (280.25,-0.57) -- (259.29,-29.42) ;
\draw    (259.29,-29.42) -- (238.32,-0.57) ;
\draw    (238.32,-0.57) -- (204.4,10.46) ;
\draw    (259.29,28.29) -- (225.37,39.31) ;
\draw    (204.4,10.46) -- (225.37,39.31) ;
\draw    (238.32,-0.57) -- (259.29,28.29) ;
\draw    (280.25,-0.57) -- (259.29,28.29) ;
\draw    (314.17,10.46) -- (293.21,39.31) ;
\draw    (259.29,28.29) -- (293.21,39.31) ;
\draw    (280.25,-0.57) -- (314.17,10.46) ;
\draw    (293.21,39.31) -- (259.29,50.33) ;
\draw    (259.29,50.33) -- (225.37,39.31) ;
\draw  [fill={rgb, 255:red, 0; green, 0; blue, 0 }  ,fill opacity=1 ] (64.48,10.54) .. controls (64.48,9.14) and (65.61,8.01) .. (67,8.01) .. controls (68.39,8.01) and (69.52,9.14) .. (69.52,10.54) .. controls (69.52,11.93) and (68.39,13.06) .. (67,13.06) .. controls (65.61,13.06) and (64.48,11.93) .. (64.48,10.54) -- cycle ;
\draw  [fill={rgb, 255:red, 0; green, 0; blue, 0 }  ,fill opacity=1 ] (201.88,10.46) .. controls (201.88,9.06) and (203.01,7.93) .. (204.4,7.93) .. controls (205.8,7.93) and (206.93,9.06) .. (206.93,10.46) .. controls (206.93,11.85) and (205.8,12.98) .. (204.4,12.98) .. controls (203.01,12.98) and (201.88,11.85) .. (201.88,10.46) -- cycle ;
\draw  [fill={rgb, 255:red, 0; green, 0; blue, 0 }  ,fill opacity=1 ] (256.76,50.33) .. controls (256.76,48.94) and (257.89,47.81) .. (259.29,47.81) .. controls (260.68,47.81) and (261.81,48.94) .. (261.81,50.33) .. controls (261.81,51.73) and (260.68,52.86) .. (259.29,52.86) .. controls (257.89,52.86) and (256.76,51.73) .. (256.76,50.33) -- cycle ;
\draw  [fill={rgb, 255:red, 0; green, 0; blue, 0 }  ,fill opacity=1 ] (311.65,10.46) .. controls (311.65,9.06) and (312.78,7.93) .. (314.17,7.93) .. controls (315.57,7.93) and (316.7,9.06) .. (316.7,10.46) .. controls (316.7,11.85) and (315.57,12.98) .. (314.17,12.98) .. controls (312.78,12.98) and (311.65,11.85) .. (311.65,10.46) -- cycle ;
\draw  [fill={rgb, 255:red, 0; green, 0; blue, 0 }  ,fill opacity=1 ] (256.76,-29.42) .. controls (256.76,-30.81) and (257.89,-31.94) .. (259.29,-31.94) .. controls (260.68,-31.94) and (261.81,-30.81) .. (261.81,-29.42) .. controls (261.81,-28.03) and (260.68,-26.9) .. (259.29,-26.9) .. controls (257.89,-26.9) and (256.76,-28.03) .. (256.76,-29.42) -- cycle ;
\draw  [fill={rgb, 255:red, 0; green, 0; blue, 0 }  ,fill opacity=1 ] (455.86,10.1) .. controls (455.86,8.71) and (456.99,7.58) .. (458.38,7.58) .. controls (459.78,7.58) and (460.91,8.71) .. (460.91,10.1) .. controls (460.91,11.5) and (459.78,12.62) .. (458.38,12.62) .. controls (456.99,12.62) and (455.86,11.5) .. (455.86,10.1) -- cycle ;
\draw [line width=0.75]    (84.47,3.43) -- (77.43,8.5) ;
\draw [shift={(75,10.25)}, rotate = 324.22] [fill={rgb, 255:red, 0; green, 0; blue, 0 }  ][line width=0.08]  [draw opacity=0] (5.36,-2.57) -- (0,0) -- (5.36,2.57) -- (3.56,0) -- cycle    ;
\draw [line width=0.75]    (274.38,41.63) -- (266.13,44.28) ;
\draw [shift={(263.27,45.19)}, rotate = 342.22] [fill={rgb, 255:red, 0; green, 0; blue, 0 }  ][line width=0.08]  [draw opacity=0] (5.36,-2.57) -- (0,0) -- (5.36,2.57) -- (3.56,0) -- cycle    ;
\draw [line width=0.75]    (220.38,22.87) -- (215.31,15.83) ;
\draw [shift={(213.56,13.4)}, rotate = 54.22] [fill={rgb, 255:red, 0; green, 0; blue, 0 }  ][line width=0.08]  [draw opacity=0] (5.36,-2.57) -- (0,0) -- (5.36,2.57) -- (3.56,0) -- cycle    ;
\draw [line width=0.75]    (252.1,-11.73) -- (257.22,-18.73) ;
\draw [shift={(258.99,-21.15)}, rotate = 126.22] [fill={rgb, 255:red, 0; green, 0; blue, 0 }  ][line width=0.08]  [draw opacity=0] (5.36,-2.57) -- (0,0) -- (5.36,2.57) -- (3.56,0) -- cycle    ;
\draw [line width=0.75]    (294.86,9.17) -- (303.09,11.88) ;
\draw [shift={(305.94,12.81)}, rotate = 198.22] [fill={rgb, 255:red, 0; green, 0; blue, 0 }  ][line width=0.08]  [draw opacity=0] (5.36,-2.57) -- (0,0) -- (5.36,2.57) -- (3.56,0) -- cycle    ;
\draw [line width=0.75]    (467.41,-4.13) -- (464.7,4.1) ;
\draw [shift={(463.77,6.95)}, rotate = 288.22] [fill={rgb, 255:red, 0; green, 0; blue, 0 }  ][line width=0.08]  [draw opacity=0] (5.36,-2.57) -- (0,0) -- (5.36,2.57) -- (3.56,0) -- cycle    ;
\draw [line width=0.75]    (566.62,25.44) -- (566.65,16.77) ;
\draw [shift={(566.66,13.77)}, rotate = 90.22] [fill={rgb, 255:red, 0; green, 0; blue, 0 }  ][line width=0.08]  [draw opacity=0] (5.36,-2.57) -- (0,0) -- (5.36,2.57) -- (3.56,0) -- cycle    ;
\draw [line width=0.75]    (591.51,0.3) -- (596.58,7.33) ;
\draw [shift={(598.33,9.76)}, rotate = 234.22] [fill={rgb, 255:red, 0; green, 0; blue, 0 }  ][line width=0.08]  [draw opacity=0] (5.36,-2.57) -- (0,0) -- (5.36,2.57) -- (3.56,0) -- cycle    ;
\draw [line width=0.75]    (590.66,56.61) -- (585.53,63.6) ;
\draw [shift={(583.76,66.02)}, rotate = 306.22] [fill={rgb, 255:red, 0; green, 0; blue, 0 }  ][line width=0.08]  [draw opacity=0] (5.36,-2.57) -- (0,0) -- (5.36,2.57) -- (3.56,0) -- cycle    ;
\draw [line width=0.75]    (142,9.95) -- (179,10) ;
\draw [shift={(182,10)}, rotate = 180.07] [fill={rgb, 255:red, 0; green, 0; blue, 0 }  ][line width=0.08]  [draw opacity=0] (5.36,-2.57) -- (0,0) -- (5.36,2.57) -- (3.56,0) -- cycle    ;
\draw [line width=0.75]    (502,9.95) -- (539,10) ;
\draw [shift={(542,10)}, rotate = 180.07] [fill={rgb, 255:red, 0; green, 0; blue, 0 }  ][line width=0.08]  [draw opacity=0] (5.36,-2.57) -- (0,0) -- (5.36,2.57) -- (3.56,0) -- cycle    ;

\end{tikzpicture}